\documentclass[11pt]{article}
\usepackage[utf8]{inputenc}
\usepackage{amscd,amsmath,amsthm,amssymb,amsfonts}
\usepackage{txfonts}
\usepackage{eucal}
\usepackage{bbm}
\usepackage{xcolor}
\usepackage{appendix}
\usepackage{authblk}
\usepackage{graphicx} 
\usepackage{hyperref}
\hypersetup{
    colorlinks=true,
    linkcolor=blue,
    filecolor=blue,      
    urlcolor=blue,
    citecolor=blue,
    }
\usepackage{algorithm}
\usepackage{mathrsfs}
\usepackage{derivative}
\usepackage[noend]{algorithmic}
\usepackage[a4paper]{geometry}
\usepackage{tikz-cd}
\usepackage{mathtools}
 \usepackage[nottoc]{tocbibind}    
\usepackage{empheq}
\usepackage{graphics}
\usepackage{indentfirst}
\usepackage[
backend=bibtex,
style=numeric-comp,
maxnames=4
]{biblatex}
\usepackage{color}
\usepackage{tabularx,colortbl}
\numberwithin{equation}{section}

\usepackage{fancyhdr}

\usepackage{subcaption}
\DeclareCaptionLabelFormat{custom}
{%
      \textbf{#1 #2}
}
\DeclareCaptionLabelSeparator{custom}{ }
\DeclareCaptionFormat{custom}
{%
    #1 #2 #3
}
\usepackage{etoolbox}

\makeindex

\def\s{$\mbox{\c{s}}$}

\def\opn#1#2{\def#1{\operatorname{#2}}} 

\opn\Ker{Ker} \opn\Coker{Coker}  \opn\Hom{Hom} \opn\Im{Im}
\opn\End{End} \opn\Aut{Aut} \opn\defect{def} \opn\ord{ord}
\opn\id{id} \opn\dim{dim} \opn\det{det} \opn\tr{tr} \opn\grad{grad} \opn\lcm{lcm}
\opn\min{min} \opn\max{max} 
\opn\Span{Span}   \opn\rang{rang}  \opn\id{id} \opn\Ass{Ass} \opn\Min{Min}
\opn\GL{GL} \opn\SL{SL} \opn\mod{mod} \opn\diag{diag}
\opn\min{min} \opn\sgn{sgn} \opn\ini{in_<}  \opn\Mon{Mon} \opn\LC{LC_<} \opn\Hom{Hom} \opn\Ext{Ext} \opn\gini{gin_{<_{rev}}} \opn\gin{gin_{<}}
\opn\LT{LT_<}
\opn\s{supp} \opn\Tor{Tor} \opn\link{link} \opn\depth{depth} \opn\pd{pd} \opn\reg{reg} 

\newcommand{\der}{\textup{d}}

\date{}
\title{A note on instantaneous gelation for coagulation kernels vanishing on the diagonal}
\author[a,1]{Iulia Cristian}
\author[b,2]{Barbara Niethammer}
\author[b,3]{Juan J. L. Vel\'{a}zquez}
\affil[a]{Laboratoire Jacques-Louis Lions, Sorbonne University, 4 Place Jussieu, 75005 Paris, France}
\affil[b]{Institute for Applied Mathematics, University of Bonn, Endenicher Allee 60, 53115 Bonn, Germany}
\affil[1]{\href{mailto:cristian@iam.uni-bonn.de}{iulia.cristian@sorbonne-universite.fr}}
\affil[2]{\href{mailto:niethammer@iam.uni-bonn.de}{niethammer@iam.uni-bonn.de}}
\affil[3]{\href{mailto:velazquez@iam.uni-bonn.de}{velazquez@iam.uni-bonn.de}}

\begin{document}
\maketitle
\begin{abstract}
 We prove that instantaneous gelation (i.e., instantaneous loss of mass) occurs for coagulation equations with sum-type kernels of homogeneity greater than one which vanish on the diagonal. Our proof includes solutions that are Radon measures if we exclude the case of initial data that are a single Dirac delta.   


\end{abstract}

\textbf{Keywords:} coagulation equations, instantaneous gelation, kernels vanishing on the diagonal.

\textbf{MSC codes:} 34A34, 45J05, 82C22, 82D60.

\newtheorem{teo}{Theorem}[section]
\newtheorem{ex}[teo]{Example}
\newtheorem{prop}[teo]{Proposition}
\newtheorem{obss}[teo]{Observations}
\newtheorem{cor}[teo]{Corollary}
\newtheorem{lem}[teo]{Lemma}
\newtheorem{prob}[teo]{Problem}
\newtheorem{conj}[teo]{Conjecture}
\newtheorem{exs}[teo]{Examples}
\newtheorem{alg}[teo]{\bf Algorithm}
\newtheorem{assumption}[teo]{Assumption}
\theoremstyle{definition}
\newtheorem{defi}[teo]{Definition}

\theoremstyle{remark}
\newtheorem{rmk}[teo]{Remark}
\newtheorem{ass}[teo]{Assumption}

\parindent0mm




\section{Introduction}
\subsection{Background}
It  is well-known that gelation (loss of mass) occurs for the standard coagulation equation, 
\begin{align}\label{non-existence model}
    \partial_{t}f(v,t)=\frac{1}{2}\int_{(0,v)}K(v-v',v')f(v-v',t)f(v',t)\der v'-\int_{(0,\infty)}K(v,v')f(v,t)f(v',t)\der v',
\end{align}
when the coagulation kernel behaves like a power law of homogeneity $\gamma>1$ (see, for example \cite{gelationpaper,papergelationlaurencot, gelfournier2025}). In particular, for sum kernels of the form
\begin{align}\label{sum kernel}
K(v,v')=v^{\gamma}+v'^{\gamma},
\end{align}
with $\gamma>1$, gelation happens instantaneously. Actually, making use of this property, one can prove that solutions which belong to $L^1$ for the standard coagulation equation do not exist for any positive time for kernels as in \eqref{sum kernel} (see \cite{ballcarr,Carr1992,Dongen,bookcoagulation}). 
Our main result is to extend the instantaneous gelation phenomenon to hold for sum kernels of homogeneity larger than one which vanish on the diagonal, i.e., $K(v,v)=0$ for all $v\in(0,\infty)$, and to the Radon measures context. For more details, see Assumption \ref{assinstgelation}. Moreover, our result is valid for coagulation kernels that behave like sum kernels of homogeneity greater than one only for large values of $v$ and $v'$.

The kernels that constitute the main subject of study for gelation theory are kernels of the form 
(\ref{sum kernel}) with $\gamma>1$ (for which solutions of (\ref{non-existence model}) do not exist) or product kernels (for which solutions of (\ref{non-existence model}) exist, but they do not conserve mass). For a more detailed description, see the end of \cite[Section 9.2]{bookcoagulation} and \cite{gelationpaper}. Kernels which vanish on the diagonal of homogeneity $\gamma>1$, while having a physical interest, see \cite{raininitiation, bubblesdropsparticles, precipitation, physicspaper}, have not yet been studied in the mathematical literature to the best of our knowledge.

A first observation is that Dirac measures of the form $f(t)\equiv M\delta_{V}$, for some $M,V\in(0,\infty)$, for all $t\in[0,\infty)$ are mass-conserving solutions to \eqref{non-existence model} when the kernel vanishes on the diagonal. Therefore, if at $t=0$ we have a Dirac measure, i.e., if $f(0)= M\delta_{V}$, for some $M,V\in(0,\infty)$, then there exist mass-conserving solutions to \eqref{non-existence model}.

However, if $f(0)\neq M\delta_{V}$, for some $M,V\in(0,\infty)$, we can prove positivity of solutions everywhere for positive times and then adapt existing methods used in the study of coagulation equations in order to prove that instantaneous gelation occurs.

The motivation behind our result is that in \cite{physicspaper} a space-dependent coagulation equation is analyzed in order to model the onset of rain and the behavior of air bubbles in water which move due to buoyancy. Here spherical particles of volume $v$ move in space vertically and merge when their trajectories cross. This leads to a spatially inhomogeneous coagulation model for the density of  particles of size $v$ at the point in space $x$ with a  so-called differential sedimentation  kernel of the form
\begin{align}\label{rain kernel}
    K(v,v')=|v^{\frac{2}{3}}-v'^{\frac{2}{3}}|(v^{\frac{1}{3}}+v'^{\frac{1}{3}})^{2}.
\end{align}


The reason the kernel in \eqref{rain kernel} vanishes on the diagonal line $\{v=v'\}$ is because the merging of the particles is due to sedimentation and particles with the same volume fall at the same speed. This opens the question of whether we can derive some information from the study of solutions of the coagulation model with kernels of homogeneity greater than one which vanish on the diagonal in the spatially homogeneous case \eqref{non-existence model} or whether such solutions do not exist.

Our main goal in this paper is thus to show
that there is no mass conserving solution  for any positive time for \eqref{non-existence model} with \eqref{rain kernel}. This provides a mathematical motivation for the further study of the space-dependent model presented in  \cite{physicspaper}. Mass-conserving solutions of the spatially inhomogeneous model are then discussed in \cite{sedimentationkernelrainmodel}.


\subsection{Main result - Instantaneous gelation}

\begin{assumption}\label{assinstgelation}
We assume that the kernel $K\colon [0,\infty)\times [0,\infty) \to [0,\infty)$ satisfies
\begin{align}\label{rain kernel general}
    K(v,v')=h(v,v')g\Big(\frac{v}{v+v'}\Big),
\end{align}
where $g \in \textup{C}([0,\infty))$ and $h\in\textup{C}([0,\infty)^2)$ satisfy the following properties.
The function  $h$ is symmetric and there exist $H>1$  and $H_{0}>0$ such that for some $\gamma >1$ there holds
    \begin{align}
H_{0}(v^{\gamma}+v'^{\gamma})\leq h(v,v') \qquad \mbox{ if } \qquad v+v'>H
\label{constant H}
\end{align}
and 
\begin{align}\label{positivity of kernel}
\min_{v,v'\in[\frac{1}{R},R]}h(v,v')\geq C_{R}>0, \textup{ for every } R>0.
\end{align}
Additionally, there exists a constant $H_{1}>0$ such that
\begin{align}\label{upper bound small h}
 h(v,v')\leq H_{1}(v^{\gamma}+v'^{\gamma}),
 \end{align}
for all $v,v'\in[0,\infty)$. 

The function $g$ satisfies $g(x)=g(1-x)$, $g(\frac{1}{2})=0$ and 
\begin{align}\label{definition of g}
G_{0}\big|x-\tfrac{1}{2}\big|^{k}\leq g(x)\leq G_{1},
\end{align}
for some constants $G_{0},G_{1}>0$, for all $x\in[0,1]$ and for some $k\geq 1$. 
\end{assumption}

\begin{rmk}One can see that kernels of the form (\ref{rain kernel general}) include the particular case of kernels of the form 
\eqref{rain kernel} and kernels of the form 
\begin{align*}
    K(v,v')=|v^{d_{1}}-v'^{d_{1}}|(v^{d_{2}}+v'^{d_{2}}) \qquad \mbox{ or } \qquad 
    K(v,v')=|v-v'|^{d_{1}}(v^{d_{2}}+v'^{d_{2}}),
\end{align*}
for some $d_{1},d_{2}\geq 0$ with $d_{1}+d_{2}>1$.
\end{rmk}
\subsubsection*{Some notation}
\begin{itemize}
    \item We will denote by  $\mathscr{M}_{+}((0,\infty))$  the space of non-negative Radon measures
and by \begin{align*}
    \mathscr{M}_{+}((0,\infty),1+v^{\beta})&:=\bigg \{f\in\mathscr{M}_{+}((0,\infty))| \int_{(0,\infty)}(1+v^{\beta})f(v)\der v<\infty\bigg\}.
\end{align*}
    \item Here and in the following, with some abuse of notation, we write $f(v)\der v$ for Radon measures, independently of the fact whether the measure is absolutely continuous with respect to the Lebesgue measure.
    \item We use the notation $C(a_{1}\ldots,a_{n})$ for a constant depending on the set of variables $(a_{1},\ldots,a_{n})$, for $n\in\mathbb{N}$, which may differ from line to line.
\end{itemize}

\begin{defi}\label{definition solution}
Let $T\in(0,\infty]$. We call $f\in\textup{C}([0,T];\mathscr{M}_{+}((0,\infty),1+v^{\gamma}))$ a weak solution of (\ref{non-existence model}) with kernel (\ref{rain kernel general}) if
\begin{align}\label{upper bound for solution}
\sup_{t\in[0,T]}\int_{(0,\infty)}(1+v^{\gamma})f(v,t)\der v\leq C(T)
\end{align}
 and 
\begin{align}
  \int_{(0,\infty)}f(v,t)\varphi(v)\der v&- \int_{(0,\infty)}f_{\textup{in}}(v)\varphi(v)\der v\nonumber\\
  &=\frac{1}{2}\int_{0}^{t}\int_{(0,\infty)}\int_{(0,\infty)}K(v,v')\Theta_{\varphi}(v,v')f(v,s)f(v',s)\der v\der v' \der s,\label{weak formulation equation solution}
\end{align}
for every $\varphi\in\textup{C}_{\textup{c}}((0,\infty))$ and for every $t\in[0,T]$, where
\begin{align*}
        \Theta_{\varphi}(v,v')&:=\varphi(v+v')-\varphi(v)-\varphi(v').
\end{align*}
\end{defi}
\begin{defi}\label{defi mass conserving solution}
Let $T>0$.    We call $f\in\textup{C}([0,T];\mathscr{M}_{+}((0,\infty)))$ a mass-conserving solution of (\ref{non-existence model}) with kernel (\ref{rain kernel general}) if $f$ is as in Definition \ref{definition solution} and satisfies in addition 
    \begin{align*}
        \int_{(0,\infty)}vf(v,t)\der v=\int_{(0,\infty)}vf_{\textup{in}}(v)\der v,
    \end{align*}
    for all $t\in[0,T].$
\end{defi}
\begin{rmk}\label{remark intro mass cons dirac}
Note that for  $M,V\in(0,\infty)$ the measure $f=M\delta_{V}$ is a mass-conserving solution for (\ref{non-existence model}) with $K$ as in  (\ref{rain kernel general}). Thus, in order to prove instantaneous gelation we need to assume that our initial data is not a single Dirac measure. We refer to \cite{paperinventiones, Escobedo2014} for positivity arguments in similar kinetic equations. 
\end{rmk}
\begin{teo}\label{theorem instantaneous gelation}
Let $K$ satisfy Assumption \ref{assinstgelation}. If $f_{\textup{in}}\neq M\delta_{V}$, for some $M,V\in(0,\infty)$, then there is no mass-conserving solution for (\ref{non-existence model}) as in Definition \ref{defi mass conserving solution} for any time interval $[0,T]$, with $T>0$.
\end{teo}
From Remark \ref{remark intro mass cons dirac} and Theorem \ref{theorem instantaneous gelation}, we immediately deduce that
\begin{cor}
    Let $K$ satisfy Assumption \ref{assinstgelation}. Equation (\ref{non-existence model}) admits mass-conserving solutions as in Definition \ref{defi mass conserving solution} if and only if $f_{\textup{in}}= M\delta_{V}$, for some $M,V\in(0,\infty)$.
\end{cor}

The proof of Theorem \ref{theorem instantaneous gelation} consists of two steps. First, we prove positivity of solutions everywhere for positive times when we exclude a single Dirac measure at the initial time via an iterative argument. We then follow the approach used in \cite[Section 9]{bookcoagulation} to show that instantaneous gelation occurs. Our arguments rely on the fact that the main contribution to equation (\ref{non-existence model}) is given by the interaction between particles of different sizes, which allows us to work with kernels which vanish on the diagonal. 
\subsection*{Discussion on solutions that do not conserve mass}
 Notice that Definition \ref{definition solution} implies that our solution needs to have finite $\gamma$ moment. This implies that we can make use of Lebesgue's dominated convergence theorem in order to prove that there is no solution for equation (\ref{non-existence model}). We refer to \cite{ballcarr}  and \cite[Section 9.2]{bookcoagulation} for a detailed proof.  More precisely, by \eqref{upper bound for solution} and using Lebesgue's dominated converge theorem, one can prove that \begin{align*}
    \int_{0}^{t}\int_{(R,\infty)}\int_{(0,\infty)}(v^{\gamma}v'+v'^{\gamma}v)f(v,s)f(v',s)\der v'\der v \der s\rightarrow 0
\end{align*}
 as $R\rightarrow\infty$. Then by Assumption \ref{assinstgelation}, we have that there exists a constant $C>0$ such that $K(v,v')\leq C(v^{\gamma}+v'^{\gamma}).$ We are now able to redo the proof of Corollary 9.2.6 in \cite{bookcoagulation}, in order to prove that
 \begin{align*}
     \int_{(0,R)}v[f(v,t)-f_{\textup{in}}]\der v\rightarrow 0
 \end{align*}
 as $R\rightarrow\infty$, which implies that, if a solution to \eqref{non-existence model} exists, it needs to be mass-conserving.
This provides a contradiction since, due to Theorem \ref{theorem instantaneous gelation}, instantaneous gelation occurs.

In other words, our result actually states that the equation \eqref{non-existence model} does not admit solutions as in Definition \ref{definition solution} at all if at initial time we exclude the case of a Dirac measure, independently of the fact that these are mass-conserving or that they do not preserve mass.

However, our definition of solutions requires a quite strong condition on the finiteness of the $\gamma$-moment of solutions, see \eqref{upper bound for solution}. An alternative way to define solutions would be to require that the term involving the coagulation kernel is finite, i.e., we have that
 \begin{align}\label{coagulation kernel term finite}
  \int_{0}^{t}\int_{(0,\infty)}\int_{(0,\infty)}K(v,v')f(v,s)f(v',s)\der v\der v' \der s<\infty,
 \end{align}
 for any $t\in[0,T].$ We do not claim here that condition \eqref{coagulation kernel term finite} implies condition \eqref{upper bound for solution}, nor that our results hold under the weaker assumption \eqref{coagulation kernel term finite}.
\section{Proof of Theorem \ref{theorem instantaneous gelation}}\label{instantaneous gelation section}

In this section we prove instantaneous gelation for kernels of the form (\ref{rain kernel general}), keeping in mind that the results apply in particular to the kernel in (\ref{rain kernel}). We first prove positivity of solutions everywhere for positive times. This is valid since, if we exclude a single Dirac measure at the initial time, we are able to find two disjoint sets $A,B$ in the space of particle volume of positive measure, i.e., $f(A,t)>0$ and $f(B,t)>0$, for all $t\geq 0$, which are sufficiently far away. We then use an iterative argument to prove that for any point $R\in(0,\infty)$ we are able to find a sufficiently large neighbourhood of $R$ which has positive measure. This holds true since, due to the form of equation \eqref{weak formulation equation solution}, interactions between particles of smaller size contribute to the formation of larger particles. Moreover, since we can work with clusters of different sizes in order to create the necessary contribution, the fact that the coagulation kernel vanishes on the diagonal does not impose problems.

We then follow the proof in \cite[Section 9.2]{bookcoagulation} (see also \cite{Carr1992,Dongen}) in order to show that the gelation time is zero. The proof relies also in this case on the fact that the main contribution is given by particles of different sizes. Thus, we are able to adapt the proof to our framework (i.e., to kernels vanishing on the diagonal) by restricting to a region far away from the diagonal. In order to avoid this region, instead of analyzing in detail the behavior of moments of order $m$, with $m>2$, as in \cite{bookcoagulation}, we need to analyze the behavior of these moments for large particles.

Notice that the case $g=G_{1}$ in (\ref{definition of g}) is the case analyzed in \cite[Section 9.2]{bookcoagulation}, while our kernel in (\ref{rain kernel general}) also contains kernels vanishing on the diagonal and thus kernels of the form (\ref{rain kernel}).

Our result states that instantaneous gelation still takes place for kernels that behave like sum kernels only near infinity and extend the results in \cite[Section 9.2]{bookcoagulation} to hold for kernels vanishing on the diagonal. Moreover, we prove our results in the context of Radon measures instead of $L^{1}$ functions.

\begin{prop}\label{initial condition is not a dirac}
    If $f_{\textup{in}}\neq M\delta_{V}$, for some $M,V\in(0,\infty)$, we have that there exist $x_{1},x_{2}\in(0,\infty)$ and $\eta_{0}\in(0,1)$ such that for every $0<\eta\leq\eta_{0}$, we have that $A:=\textup{B}(x_{1},\eta)$ and $B:=\textup{B}(x_{2},\eta)$, with $A\cap B=\emptyset$,  $\textup{dist}(A,B)\geq \eta_{0}>0$, satisfy
\begin{align}\label{sets with positive mass}
    \int_{A}f_{\textup{in}}(v)\der v>0 \textup{ and }
   \int_{B}f_{\textup{in}}(v)\der v>0. \end{align}
\end{prop}

\begin{proof}


Since $f_{\textup{in}}\neq 0,$ we have that there exists $L\in\mathbb{N}$ such that $  \int_{(0,L]}f_{\textup{in}}(v)\der v>0$. Let $n\in\mathbb{N}$. We split the interval $(0,L]$ into $2^{n}$ intervals of the form $\Big(\frac{Lj}{2^{n}},\frac{L(j+1)}{2^{n}}\Big]$, for $j\in\{0,\ldots,2^{n}-1\}$. We first prove the following claim.

\textit{Claim:} Assume that $f_{\textup{in}}$ restricted to the interval  $(0,L]$ is not a single Dirac measure. Then there exists a sufficiently large $M>1$ such that for any $n\geq M$, there exist $l_{1},l_{2}\in\mathbb{N}\cap[0,L]$, with $l_{2}\geq l_{1}+2$, such that  
\begin{align}\label{the claim is sufficient}
    \int_{\big(\frac{Ll_{1}}{2^{n}},\frac{L(l_{1}+1)}{2^{n}}\big]}f_{\textup{in}}(v)\der v>0 \textup{ and }
   \int_{\big(\frac{Ll_{2}}{2^{n}},\frac{L(l_{2}+1)}{2^{n}}\big]}f_{\textup{in}}(v)\der v>0.
\end{align}
\textit{Proof of claim:} Assume by contradiction that for any $n\in\mathbb{N}$ and any  $l_{1},l_{2}\in\mathbb{N}\cap[0,L]$, with $l_{2}\geq l_{1}+2$, it holds that either  $\int_{\big(\frac{Ll_{1}}{2^{n}},\frac{L(l_{1}+1)}{2^{n}}\big]}f_{\textup{in}}(v)\der v=0$ or 
   $\int_{\big(\frac{Ll_{2}}{2^{n}},\frac{L(l_{2}+1)}{2^{n}}\big]}f_{\textup{in}}(v)\der v=0.$ This implies in particular that for every $n\in\mathbb{N}$ there exists a point $l_{n}$, which depends on $n\in\mathbb{N}$, such that
\begin{align}\label{this will converge to a dirac for the contradiction}
      \int_{\big(\frac{Ll_{n}}{2^{n}},\frac{L(l_{n}+2)}{2^{n}}\big]}f_{\textup{in}}(v)\der v=  \int_{(0,L]}f_{\textup{in}}(v)\der v.
\end{align}
We iterate this reasoning. When $n=2$ it holds that there exists $j_{2}\in\{0,\ldots,2\}$ such that $ \int_{\big(\frac{Lj_{2}}{2^{2}},\frac{L(j_{2}+2)}{2^{2}}\big]}f_{\textup{in}}(v)\der v=  \int_{(0,L]}f_{\textup{in}}(v)\der v.$ We now split the interval $\big(\frac{Lj_{2}}{2^{2}},\frac{L(j_{2}+2)}{2^{2}}\big]$ into $4$ subintervals and apply the same argument. We obtain that there exists $j_{3}\in\{2j_{2},\ldots, 2(j_{2}+1)\}$ such that
\begin{align*}
     \int_{\big(\frac{Lj_{3}}{2^{3}},\frac{L(j_{3}+2)}{2^{3}}\big]}f_{\textup{in}}(v)\der v=    \int_{\big(\frac{Lj_{2}}{2^{2}},\frac{L(j_{2}+2)}{2^{2}}\big]}f_{\textup{in}}(v)\der v= \int_{(0,L]}f_{\textup{in}}(v)\der v.
\end{align*}
Inductively, we find that for every $n\in\mathbb{N}$ there exists $j_{n}\in\{2j_{n-1},\ldots, 2(j_{n-1}+1)\}$ such that 
\begin{align*}
      \int_{\big(\frac{Lj_{n}}{2^{n}},\frac{L(j_{n}+2)}{2^{n}}\big]}f_{\textup{in}}(v)\der v=  \int_{(0,L]}f_{\textup{in}}(v)\der v.
\end{align*}

If there exists $i\in\mathbb{N}$ such that $j_{n}=0$, for every $n\geq i$, then by the right-continuity of the cumulative distribution it holds that $  0=\lim_{n\rightarrow\infty}\int_{\big(0,\frac{2L}{2^{n}}\big]}f_{\textup{in}}(v)\der v=  \int_{(0,L]}f_{\textup{in}}(v)\der v$, which is a contradiction. As such, we have that $j_{n}\neq 0$, for sufficiently large $n$.

Thus, it further holds that
\begin{align*}
   \int_{\big[\frac{Lj_{n}}{2^{n}},\frac{L(j_{n}+2)}{2^{n}}\big]}f_{\textup{in}}(v)\der v= \int_{\big(\frac{Lj_{n}}{2^{n}},\frac{L(j_{n}+2)}{2^{n}}\big]}f_{\textup{in}}(v)\der v=  \int_{(0,L]}f_{\textup{in}}(v)\der v. 
\end{align*} 
For every $n\in\mathbb{N}$, $n\geq 2$, we denote by $\alpha_{n}:=\frac{Lj_{n}}{2^{n}}$ and by $\beta_{n}:=\frac{L(j_{n}+2)}{2^{n}}$. We thus have that $f_{\textup{in}}\big((0,\alpha_{n})\big)=0$ and that  $f_{\textup{in}}\big((\beta_{n},L]\big)=0$.  In addition, by the way we constructed $\alpha_{n},\beta_{n}$, we have that $\{\alpha_{n}\}_{n}$ is an increasing sequence, $\{\beta_{n}\}_{n}$ is a decreasing sequence, and that $|\alpha_{n}-\beta_{n}|\leq 2^{1-n}$.

Together with the fact that $j_{n}\neq 0$ for sufficiently large $n$, this implies that there exists a point $x_{0}\in(0,L]$ such that $\alpha_{n}\rightarrow x_{0}$ and $\beta_{n}\rightarrow x_{0}$. Since $(0,x_{0})=\bigcup_{n=1}^{\infty}(0,\alpha_{n})$, it holds that $f_{\textup{in}}((0,x_{0}))\leq \sum_{n=1}^{\infty}f_{\textup{in}}((0,\alpha_{n}))=0$. In the same manner, we have that $f_{\textup{in}}\big((x_{0},L]\big)=0.$ In other words, we obtain that $f_{\textup{in}}(\{x_{0}\})=  \int_{(0,L]}f_{\textup{in}}(v)\der v$, which contradicts the assumption that $f_{\textup{in}}$ is not a single Dirac measure on $(0,L]$ and this proves our claim.

We now use the claim in order to prove the statement of the proposition. If \eqref{the claim is sufficient} holds, then there is nothing to prove. Otherwise it holds that there exists $x_{0}\in(0,L]$ such that $f_{\textup{in}}(\{x_{0}\})=  \int_{(0,L]}f_{\textup{in}}(v)\der v$. Since by assumption $f_{\textup{in}}$ is not a single Dirac measure, we have that there exists $L_{1}>L$ such that  $  \int_{(L,L_{1}]}f_{\textup{in}}(v)\der v>0$. Using the same argument as before, we either have that the statement of the proposition is true or there exists $x_{1}\in(L,L_{1}]$ such that $f_{\textup{in}}(\{x_{1}\})=  \int_{(L,L_{1}]}f_{\textup{in}}(v)\der v$. Thus, we have that $f_{\textup{in}}$ restricted to the interval $(0,L_{1}]$ is a sum of two Dirac measures and the proposition holds.
\end{proof}
\begin{assumption}\label{assumption x2 x1}
   Let $x_{1},x_{2}$ be as in Proposition \ref{initial condition is not a dirac}.  Without loss of generality, we assume in all the following that $x_{2}>x_{1}$ for simplicity of notation.
\end{assumption}
We first prove that in order not to have solutions which are Dirac measures, it suffices to exclude the case of Dirac measures at initial time.
\begin{prop}[Positivity]\label{positivity}
  Assume $f$ is a mass-conserving solution  for (\ref{non-existence model}) as in Definition \ref{defi mass conserving solution} with $K$ as in (\ref{rain kernel general}). Assume $f_{\textup{in}}\neq M\delta_{V}$, for some $M,V\in(0,\infty)$. Then we have that, for every $t\in(0,\infty)$ and $R>0$, there exists a constant $C(t,R)>0$, depending on $t$ and $R$, such that
\begin{align}\label{constant t r all t all r}
    \int_{R}^{\infty}f(v,t)\der v\geq C(t,R).
\end{align}
\end{prop}
We postpone the proof of Proposition \ref{positivity} to after Proposition \ref{induction step}.  Proposition \ref{induction step} is a simplified version of Proposition  \ref{positivity}, which gives a clear picture on how the coagulation mechanism drives the size of the particles to increase, allowing us to prove positivity in a suitable sense.
\begin{prop}[Induction basis]\label{induction step}
  Assume $f$ is a mass-conserving solution  for (\ref{non-existence model}) as in Definition \ref{defi mass conserving solution} with $K$ as in (\ref{rain kernel general}). Assume $f_{\textup{in}}\neq M\delta_{V}$, for some $M,V\in(0,\infty)$. Let $A,B$, $x_{1},x_{2}$ and $\eta_{0}$ be as in Proposition \ref{initial condition is not a dirac}. Then we have that, for every  $t\in(0,\infty)$, there exists a constant $C(t,x_{1},x_{2},\eta_{0})>0$, depending on $t$, $x_{1},x_{2}$, and $\eta_{0}$, such that
\begin{align*}
    \int_{\textup{B}\big(x_{1}+x_{2},\frac{5\eta_{0}}{2}\big)}f(v,t)\der v \geq C(t,x_{1},x_{2},\eta_{0}).
\end{align*}
\end{prop}
\begin{proof}
Let $A=\textup{B}(x_{1},\frac{\eta_{0}}{2})$, $B=\textup{B}(x_{2},\frac{\eta_{0}}{2})$, where $\eta_{0}$ is as in Proposition \ref{initial condition is not a dirac}.
  Let $\varphi\in\textup{C}_{0}(0,\infty)$ in (\ref{weak formulation equation solution}) be such that 
\begin{align}
    \varphi(v)&=1, \textup{ if } v\in A;\label{statement positivity 1}\\
    \varphi(v)&=0, \textup{ if } \textup{dist}(v,A)\geq \frac{\eta_{0}}{2}.\label{statement positivity 2}
\end{align}
Then, since $\varphi(v+v')\geq 0$ and by the symmetry of $K(v,v')$, we have that
\begin{align*}
    \int_{(0,\infty)}&\varphi(v)  f(v,t)\der v
    \\
    &\geq \int_{(0,\infty)}\varphi(v)f_{\textup{in}}(v)\der v-\int_{0}^{t}\int_{(0,\infty)}\int_{(0,\infty)}f(v,s)f(v',s)K(v,v')\varphi(v)\der v' \der v\der s.
\end{align*}
Notice that from (\ref{upper bound small h}) and (\ref{definition of g}), it follows that there exists a constant $C>0$ such that $K(v,v')\leq C (v^{\gamma}+v'^{\gamma})$. Thus
\begin{align*}
&\int_{(0,\infty)}\varphi(v)f(v,t)\der v\\
&\geq \int_{(0,\infty)}\varphi(v)f_{\textup{in}}(v)\der v-C\int_{0}^{t}\int_{(0,\infty)}\int_{(0,\infty)}f(v,s)f(v',s)(v^{\gamma}+v'^{\gamma})\varphi(v)\der v' \der v\der s\\
&\geq \int_{(0,\infty)}\varphi(v)f_{\textup{in}}(v)\der v-C\int_{0}^{t}\bigg(\int_{(0,\infty)} (1+v'^{\gamma})f(v',s)\der v' \int_{(0,\infty)}(1+v^{\gamma})f(v,s)\varphi(v) \der v\bigg) \der s.
    \end{align*}
    From (\ref{upper bound for solution}), there exists $C(t)>0$, which is non-decreasing in $t$, such that $\sup_{s\in[0,t]}\int_{(0,\infty)}(1+v^{\gamma})f(v,t)\der v\leq C(t)$. Moreover, using the fact that $\varphi$ is supported in $\textup{B}(x_{1},\eta_{0})$, we can deduce that there exists a constant $C(t,x_{1},\eta_{0})>0$ such that
    \begin{align}\label{integral form ode}
\int_{(0,\infty)}\varphi(v)f(v,t)\der v&\geq \int_{(0,\infty)}\varphi(v)f_{\textup{in}}(v)\der v-C(t,x_{1},\eta_{0})\int_{0}^{t}\int_{(0,\infty)}f(v,s)\varphi(v) \der v\der s.
    \end{align}
    We notice that \eqref{integral form ode}  is the integral form of a differential inequality whose integration yields that there exists a constant $C(t,x_{1},\eta_{0})>0$ such that 
        \begin{align*}
\int_{(0,\infty)}\varphi(v)f(v,t)\der v&\geq C(t,x_{1},\eta_{0})\int_{(0,\infty)}\varphi(v)f_{\textup{in}}(v)\der v\\
&\geq C(t,x_{1},\eta_{0})\int_{A}f_{\textup{in}}(v)\der v\geq C(t,x_{1},\eta_{0}),
    \end{align*}
where we used (\ref{sets with positive mass}) for the last inequality. It follows that
         \begin{align}\label{lower bound for A}
\int_{\textup{B}(x_{1},\eta_{0})}f(v,t)\der v\geq \int_{(0,\infty)}\varphi(v)f(v,t)\der v\geq C(t,x_{1},\eta_{0}).
    \end{align}  
    Using the same argument for the set $B$, we obtain that there exists a constant $C(t,x_{2},\eta_{0})>0$ such that
     \begin{align}\label{lower bound for B}
\int_{\textup{B}(x_{2},\eta_{0})}f(v,t)\der v\geq C(t,x_{2},\eta_{0}).
    \end{align} 
    Taking now
\begin{align*}
    \varphi(v)&=1, \textup{ if } v\in \textup{B}(x_{1}+x_{2},2\eta_{0});\\
    \varphi(v)&=0, \textup{ if } \textup{dist}(v,\textup{B}(x_{1}+x_{2},2\eta_{0}))\geq \frac{\eta_{0}}{2},
\end{align*}
we obtain, as before, that
\begin{align*}
    \int_{(0,\infty)}\varphi(v)f(v,t)\der v&\geq \int_{(0,\infty)}\varphi(v)f_{\textup{in}}(v)\der v\\
&-\int_{0}^{t}\int_{(0,\infty)}\int_{(0,\infty)}f(v,s)f(v',s)K(v,v')\varphi(v)\der v' \der v\der s\\
&+\frac{1}{2}\int_{0}^{t}\int_{(0,\infty)}\int_{(0,\infty)}\varphi(v+v')f(v,s)f(v',s)K(v,v')\der v' \der v\der s\\
    &\geq-C\int_{0}^{t}\int_{(0,\infty)}(1+v'^{\gamma})f(v',s)\der v'\int_{(0,\infty)}(1+v^{\gamma})f(v,s)\varphi(v) \der v\der s\\
&+\frac{1}{2}\int_{0}^{t}\int_{\textup{B}(x_{1},\eta_{0})}\int_{\textup{B}(x_{2},\eta_{0})}\varphi(v+v')f(v,s)f(v',s)K(v,v')\der v' \der v\der s.
\end{align*}
 By Proposition \ref{initial condition is not a dirac}, if $v\in \textup{B}(x_{1},\eta_{0})$ and $v'\in \textup{B}(x_{2},\eta_{0})$, we have that $\textup{dist}(v,v')\geq \eta_{0}$ and that $v+v'\leq x_{1}+x_{2}+2\eta_{0}$. In other words, from (\ref{definition of g}), it holds that there exists a constant $C(x_{1},x_{2},\eta_{0})$, depending on $x_{1},x_{2}$ and $\eta_{0}$, such that
\begin{align}\label{lower bound on g}
    g\Big(\frac{v}{v+v'}\Big)\geq \frac{ G_{0}|v-v'|^{k}}{2^{k}(v+v')^{k}}\geq  C(x_{1},x_{2},\eta_{0}),
\end{align}
where $g$ is as in (\ref{rain kernel general}). From  (\ref{rain kernel general}), (\ref{upper bound for solution}), (\ref{lower bound on g}) and since $\textup{supp}(\varphi)\subseteq \textup{B}(x_{1}+x_{2}, \frac{5\eta_{0}}{2})$, we further deduce that
\begin{align*}
    \int_{(0,\infty)}\varphi(v)f(v,t)\der v&\geq-C(t,x_{1},x_{2},\eta_{0})  \int_{0}^{t}\int_{(0,\infty)}\varphi(v)f(v,s)\der v\der s\\
&+C(x_{1},x_{2},\eta_{0})\int_{0}^{t}\int_{\textup{B}(x_{1},\eta_{0})}\int_{\textup{B}(x_{2},\eta_{0})}\varphi(v+v')f(v,s)f(v',s)h(v,v')\der v' \der v\der s.
\end{align*}
Denote by $C_{H,x_{1},x_{2},\eta_{0}}=\min\{C_{R}, C_{H}\}$, where $C_{R}$ is the constant associated to $R=\max\{x_{2}+\eta_{0}, \frac{1}{x_{1}-\eta_{0}}\}$ in (\ref{positivity of kernel}), and $C_{H}:=2^{1-\gamma}H^{\gamma}H_{0}$, where $H$ and $H_{0}$ are as in (\ref{constant H}). By Assumption \ref{assumption x2 x1}, we have that $h(v,v')\geq C_{R}$ when $v+v'\leq H$ and since $v^{\gamma}+v'^{\gamma}\geq 2^{1-\gamma}(v^{\gamma}+v'^{\gamma})$, we have that  $h(v,v')\geq C_{H}$ when $v+v'> H$. Then
\begin{align}\label{later use}
  \int_{(0,\infty)}\varphi(v)f(v,t)\der v   &\geq-C(t,x_{1},x_{2},\eta_{0})  \int_{0}^{t}\int_{(0,\infty)}\varphi(v)f(v,s)\der v\der s\nonumber\\
&+C(x_{1},x_{2},\eta_{0})C_{H,x_{1},x_{2},\eta_{0}}\int_{0}^{t}\int_{\textup{B}(x_{1},\eta_{0})}\int_{\textup{B}(x_{2},\eta_{0})}f(v,s)f(v',s)\der v' \der v\der s.
\end{align}
Using (\ref{lower bound for A}) and (\ref{lower bound for B}), we have that there exist $C_{1}(t,x_{1},x_{2},\eta_{0}) , C_{2}(t,x_{1},x_{2},\eta_{0}) >0$ such that
\begin{align*}
    \int_{(0,\infty)}\varphi(v)f(v,t)\der v\geq-C_{1}(t,x_{1},x_{2},\eta_{0})  \int_{0}^{t}\int_{(0,\infty)}\varphi(v)f(v,s)\der v\der s+C_{2}(t,x_{1},x_{2},\eta_{0}),
\end{align*}
so we conclude as before that there exist $C(t,x_{1},x_{2},\eta_{0}) >0$ such that
\begin{align*}
    \int_{\textup{B}\big(x_{1}+x_{2},\frac{5\eta_{0}}{2}\big)}f(v,t)\der v \geq C(t,x_{1},x_{2},\eta_{0}).
\end{align*}
This concludes our proof.
\end{proof}
\begin{rmk}\label{remark needed for induction}
Let $n\in\mathbb{N}$. By modifying the constants in \eqref{statement positivity 1}, \eqref{statement positivity 2}, we can obtain in \eqref{lower bound for A} that
    \begin{align}
    \int_{\textup{B}(x_{1},\frac{\eta_{0}}{2^{n}})}f(v,t)\der v\geq C(t,x_{1},\eta_{0},n) \textup{ and }  \int_{\textup{B}(x_{1}+x_{2},\frac{\eta_{0}}{2^{n}})}f(v,t)\der v\geq C(t,x_{1},x_{2},\eta_{0},n),
\end{align}
for some constants $C(t,x_{1},\eta_{0},n), C(t,x_{1},x_{2},\eta_{0},n)>0$.
\end{rmk}
We are now able to generalize the argument used in Proposition \ref{induction step} in order to prove Proposition \ref{positivity}.
\begin{proof}[Proof of Proposition \ref{positivity}]
We now wish to prove that (\ref{constant t r all t all r}) holds. Let $R>0$. Let $x_{1}$ and $x_{2}$ as in Proposition \ref{initial condition is not a dirac}. Then there exists $N\in\mathbb{N}$ such that $Nx_{1}+x_{2}>R+\eta_{0}$. We will prove that there exists a constant $C(t,x_{1},x_{2},\eta_{0},N)>0$ such that
\begin{align*}
\int_{\textup{B}(Nx_{1}+x_{2},\eta_{0})}f(v,t)\der v \geq C(t,x_{1},x_{2},\eta_{0},N)
    \end{align*} 
using an iterative argument.  If $N=0$ or $N=1$, the statement follows from Remark \ref{remark needed for induction}. Let $N\geq 2$. We will prove that for every $n\in\{1,\ldots,N\}$ it holds that there exists a constant $C(t,x_{1},x_{2},\eta_{0},N)>0$ such that
\begin{align}\label{overline c}
\int_{\textup{B}\big (n x_{1}+x_{2},\frac{\eta_{0}}{2^{N-n}}\big )}f(v,t)\der v\geq C(t,x_{1},x_{2},\eta_{0},N).
    \end{align} 

Moreover, by Remark \ref{remark needed for induction}, we have that
\begin{align}\label{k plus one}
    \int_{\textup{B}\big(x_{1},\frac{\eta_{0}}{2^{N+1}}\big)}f(v,t)\der v\geq C(t,x_{1},\eta_{0},N).
\end{align}
Notice that the case $n=1$ needed for the iterative argument follows from Remark \ref{remark needed for induction} and \eqref{k plus one}.

 We will only prove that if \eqref{overline c} holds for $n=N-1$, then \eqref{overline c} holds for $n=N$.  As before, let
\begin{align*}
    \varphi(v)&=1, \textup{ if } v\in \textup{B}\bigg(N x_{1}+x_{2},\Big(1-\frac{1}{2^{N}}\Big)\eta_{0}\bigg);\\
    \varphi(v)&=0, \textup{ if } \textup{dist}\bigg(v,\textup{B}\bigg(N x_{1}+x_{2},\Big(1-\frac{1}{2^{N}}\Big)\eta_{0}\bigg)\bigg)\geq \frac{\eta_{0}}{2^{{N}}}.
\end{align*}
By Assumption \ref{assumption x2 x1}, we have $x_{2}>x_{1}$. If $v\in \textup{B}((N-1)x_{1}+x_{2},\frac{\eta_{0}}{2})$ and $v'\in\textup{B}(x_{1},\frac{\eta_{0}}{2^{N}})$, then $\textup{dist}(v,v')\geq \eta_{0}>0$ since we have that $|(N-2)x_{1}+x_{2}-\frac{\eta_{0}}{2^{N}}-\frac{\eta_{0}}{2}|\geq |(N-2)x_{1}+x_{2}-\eta_{0}|\geq\eta_{0}$ by Proposition \ref{initial condition is not a dirac}. Thus, using similar arguments as in (\ref{lower bound on g}) and in the derivation of (\ref{later use}), we obtain that
\begin{align*}
    \int_{(0,\infty)}\varphi(v)f(v,t)\der v&\geq- C(t,x_{1},x_{2},\eta_{0},N) \int_{0}^{t}\int_{(0,\infty)}\varphi(v)f(v,s)\der v\der s\\
&+C(x_{1},x_{2},\eta_{0},N)\int_{0}^{t}\int_{\textup{B}(x_{1},\frac{\eta_{0}}{2^{N}})}\int_{\textup{B}((N-1)x_{1}+x_{2},\frac{\eta_{0}}{2})}h(v,v')f(v,s) f(v',s)\der v\der v' \der s\\
&\geq-C(t,x_{1},x_{2},\eta_{0},N) \int_{0}^{t}\int_{(0,\infty)}\varphi(v)f(v,s)\der v\der s\\
&+C(x_{1},x_{2},\eta_{0},N)\int_{0}^{t}\int_{\textup{B}(x_{1},\frac{\eta_{0}}{2^{N}})}f(v,s)\der v \int_{\textup{B}((N-1)x_{1}+x_{2},\frac{\eta_{0}}{2})}f(v',s)\der v' \der s\\
&\geq-C_{1}(t,x_{1},x_{2},\eta_{0},N) \int_{0}^{t}\int_{(0,\infty)}\varphi(v)f(v,s)\der v\der s+C_{2}(t,x_{1},x_{2},\eta_{0},N),
\end{align*}
for some $C_{1}(t,x_{1},x_{2},\eta_{0},N), C_{2}(t,x_{1},x_{2},\eta_{0},N)>0$, where in the last inequality we used (\ref{overline c}). Thus, it holds that there exists a constant $C(t,x_{1},x_{2},\eta_{0},N)>0$ such that
    \begin{align*}
\int_{R}^{\infty}f(v,t)\der v\geq \int_{\textup{B}(Nx_{1}+x_{2},\eta_{0})}f(v,t)\der v\geq \int_{(0,\infty)}\varphi(v)f(v,t)\der v\geq C(t,x_{1},x_{2},\eta_{0},N),
    \end{align*} 
    which concludes our proof.
\end{proof}
We can now directly conclude that the moments of order $m$ of our solutions are large for large values of $m$.
\begin{cor}\label{corollary infinity}
  Assume $f$ is a mass-conserving solution  for (\ref{non-existence model}) as in Definition \ref{defi mass conserving solution} with $K$ as in (\ref{rain kernel general}).  Assume $f_{\textup{in}}\neq M\delta_{V}$, for some $M,V\in(0,\infty)$. Then for each $t>0$ it holds that
\begin{align*} \liminf_{R\rightarrow\infty}\liminf_{m\rightarrow \infty} M_{R,m}(f(t))^{\frac{1}{m}}  :=\liminf_{R\rightarrow\infty}\liminf_{m\rightarrow \infty} \bigg(\int_{[R,\infty)}(v-R)^{m}f(v,t)\der v\bigg)^{\frac{1}{m}}=\infty.
\end{align*}
\begin{proof}
By Proposition \ref{positivity}, we have that there exists $C(t,2R)>0$ such that
\begin{align*}
M_{R,m}(f(t))^{\frac{1}{m}}  :=\bigg(\int_{[R,\infty)}(v-R)^{m}f(v,t)\der v\bigg)^{\frac{1}{m}}&\geq R \bigg(\int_{[2R,\infty)}f(v,t)\der v\bigg)^{\frac{1}{m}}\\
&\geq R C(t,2R)^{\frac{1}{m}}.
\end{align*}
As $C(t,2R)>0$, we have that $C(t,2R)^{\frac{1}{m}}\rightarrow 1$ as $m\rightarrow \infty$ and thus 
\begin{align*}
 \liminf_{m\rightarrow \infty} M_{R,m}(f(t))^{\frac{1}{m}}  \geq R.
\end{align*}
The conclusion follows by letting $R\rightarrow\infty$.
\end{proof}
\end{cor}
The theory on gelation for coagulation equations focuses on proving that solutions for coagulation equations with kernel of homogeneity $\gamma>1$ lose mass in finite time. It is thus useful to define the time when gelation occurs.
\begin{defi} We define
    \begin{align*}
        T_{\textup{gel}}:=\sup\Big\{t\geq 0: M_{1}(f(t))=M_{1}(f_{\textup{in}})\Big\}\in[0,\infty].
    \end{align*}
\end{defi}

We now follow the proof in \cite[Chapter 9]{bookcoagulation} (or in \cite{Dongen, Carr1992} for the discrete coagulation equation) in order to prove that instantaneous gelation (i.e., $T_{\textup{gel}}=0$) occurs for kernels of the form (\ref{rain kernel general}). The reason why we can adapt this proof to our case is that the main ingredient of the proof is the interaction between particles of different sizes, which allows us to include kernels which vanish on the diagonal.
\begin{prop}\label{all large moments finite}
For all $m\geq 2$ and $T\in(0,T_{\textup{gel}})$, it holds that
\begin{align}\label{momentestimate}
    \sup_{t\in[0,T]} M_{m}(f(t))<\infty.
\end{align}
\end{prop}
\begin{proof}
Define $\chi_{R}\in\textup{C}(0,\infty)$
\begin{align}
   \chi_{R}(v)=
\begin{cases}
1 ,& \textup{ if } v\leq R;  \\
0 ,&\textup{ if } v> R+1
\end{cases}
\end{align}
and denote 
\begin{align*}
    J_{R}(t):=\int_{(0,\infty)}v\chi_{R}(v)f(v,t)\der v
\qquad \mbox{ and  } \qquad 
I_{R}(t):=\int_{(0,\infty)}[1-\chi_{R}(v)]vf(v,t)\der v.
\end{align*}
Notice that 
\begin{align*}
    I_{R}(t)+J_{R}(t)=\int_{(0,\infty)}vf(v,t)\der v=\int_{(0,\infty)}vf_{\textup{in}}(v)\der v=:v_{0},
\end{align*}
for every $R>0$ and every $t\leq T_{\textup{gel}}$. Moreover, notice that
\begin{align*}
    I_{R}(s)-I_{R}(t)=J_{R}(t)-J_{R}(s),
\end{align*}
for every $0\leq s<t\leq T_{\textup{gel}}$.

Let $T<T_{\textup{gel}}$. By Dini's monotone convergence theorem, there exists $R_{0}>1$ such that
\begin{align}\label{lower bound jr}
 J_{R}(t)\geq\frac{v_{0}}{2},
\end{align}
for every $R\geq R_{0}$ and every $t\in[0,T].$

Let now $\varphi(v):=v\chi_{R}(v)$ and $0\leq s<t\leq T<T_{\textup{gel}}$. Since $\varphi\in\textup{C}_{0}(0,\infty)$, we can test with $\varphi$ in (\ref{weak formulation equation solution}) and obtain information about $I_{R}$ using the equation for $J_{R}$. More precisely, we have that
\begin{align*}
    I_{R}(s)-I_{R}(t)=J_{R}(t)-J_{R}(s)&=\int_{(0,\infty)}v\chi_{R}(v)f(v,t)\der v- \int_{(0,\infty)}v\chi_{R}(v)f(v,s)\der v\\
 &\leq \int_{s}^{t}\int_{0}^{\frac{R}{2}}\int_{R}^{\infty}\Theta_{\varphi}(v,v')K(v,v')f(v,\tau)f(v',\tau)\der v'\der v\der \tau,
\end{align*}
where in the last inequality we have used that $\varphi(v+v')-\varphi(v')-\varphi(v)\leq 0$, for all $v,v'\in(0,\infty)$. Furthermore, since we are in the region $(v,v')\in(0,\frac{R}{2})\times(R,\infty)$ and since $\chi_{R}$ is a non-increasing function, we have that
\begin{align*}
    \Theta_{\varphi}(v,v')&=(v+v')\chi_{R}(v+v')-v'\chi_{R}(v')-v\chi_{R}(v)\\
    &=v\chi_{R}(v+v')+v'[\chi_{R}(v+v')-\chi_{R}(v')]-v
    \leq v[\chi_{R}(v+v')-1]\leq  v[\chi_{R}(v')-1],
\end{align*}
which implies that
\begin{align}
  I_{R}(t)-I_{R}(s)\leq -\int_{s}^{t}\int_{0}^{\frac{R}{2}}\int_{R}^{\infty}v[1-\chi_{R}(v')] K(v,v')f(v,\tau)f(v',\tau)\der v'\der v\der \tau.\label{upper bound in ode}
\end{align}
We now analyze $K(v,v')$ in the region $(v,v')\in(0,\frac{R}{2})\times(R,\infty)$. Using the definition of $g$ in (\ref{definition of g}) and the fact that $v\leq \frac{v'}{2}$, we notice that 
\begin{align*}
   g\Big(\frac{v}{v+v'}\Big)\geq \frac{G_{0}|v'-v|^{k}}{2^{k}(v+v')^{k}}\geq \frac{G_{0}v'^{k}}{4^{k}(v+v')^{k}}\geq \frac{G_{0}}{6^{k}}.
\end{align*}
Then, fixing $R\geq H$, where $H$ is as in (\ref{constant H}), we have that $v+v'\geq H$ and thus there exists a constant $C>0$, independent of $R$, such that 
\begin{align}\label{lower bound kernel finite moment estimates}
K(v,v')\geq Ch(v,v')\geq C(v^{\gamma}+v'^{\gamma})\geq  Cv'^{\gamma}
\end{align}
in the region $(v,v')\in(0,\frac{R}{2})\times(R,\infty)$. We plug (\ref{lower bound kernel finite moment estimates}) into (\ref{upper bound in ode}) and, since $\gamma>1$, we obtain that
\begin{align*}
        I_{R}(s)-I_{R}(t)&\leq -C\int_{s}^{t}\int_{0}^{\frac{R}{2}}\int_{R}^{\infty}v[1-\chi_{R}(v')] v'^{\gamma}f(v,\tau)f(v',\tau)\der v'\der v\der \tau\\
        &\leq -C\int_{s}^{t}\int_{0}^{\frac{R}{2}}vf(v,\tau)\der v\int_{R}^{\infty}v'^{\gamma}[1-\chi_{R}(v')] f(v',\tau)\der v'\der \tau\\
             &\leq -CR^{\gamma-1}\int_{s}^{t}\int_{0}^{\frac{R}{2}}vf(v,\tau)\der v\int_{R}^{\infty}v'[1-\chi_{R}(v')] f(v',\tau)\der v'\der \tau\\
               &=-CR^{\gamma-1}\int_{s}^{t}I_{R}(\tau)\int_{0}^{\frac{R}{2}}vf(v,\tau)\der v \der \tau,
 \end{align*}
 where in the last line we used the definition of $I_{R}$. However, for $R>2(R_{0}+1)$, where $R_{0}$ was chosen in order for (\ref{lower bound jr}) to hold, we have that
  \begin{align*}
        \int_{0}^{\frac{R}{2}}vf(v,t)\der v\geq J_{\frac{R}{2}-1}(t)\geq\frac{v_{0}}{2},
 \end{align*}
for all $t\in[0,T]$, and thus
 \begin{align}\label{inequality ode}
        I_{R}(s)-I_{R}(t)\leq -CR^{\gamma-1}v_{0}\int_{s}^{t} I_{R}(\tau)\der \tau.
 \end{align}
 Using (\ref{inequality ode}), we can use the same argument as in \cite[Proof of Lemma 9.2.2]{bookcoagulation} to deduce that
 \begin{align}\label{exponential decay}
I_{R}(t)\leq v_{0} \textup{e}^{-Cv_{0}R^{\gamma-1}(T-t)}.
 \end{align}
The argument used in \cite[Proof of Lemma 9.2.2]{bookcoagulation} consists in finding a suitable ODE, from which we can derive an exponential bound, together with the fact that $I_{R}(T)\leq v_{0}$.

Thus, let $m\geq 2$ and $0\leq t<T<T_{\textup{gel}}$. Let $L>\max\{2(R_{0}+1),H\}$, with $R_{0}$ as in (\ref{lower bound jr}) and $H$ as in (\ref{constant H}). Let $l=L+d$, for some $d\in\mathbb{N}.$ We have
\begin{align*}
   \int_{L}^{l}v^{m}f(v,t)\der v&\leq\sum_{j=L}^{l-1}(j+1)^{m-1} \int_{j}^{j+1}vf(v,t)\der v\\
   &\leq\sum_{j=L}^{l-1}(j+1)^{m-1}[I_{j-1}(t)-I_{j+1}(t)]\\
      &=\sum_{j=L-1}^{l-2}(j+2)^{m-1}I_{j}(t)-\sum_{j=L+1}^{l}j^{m-1}I_{j}(t)\\
        &\leq (L+1)^{m-1}I_{L-1}(t)+(L+2)^{m-1}I_{L}(t)+\sum_{j=L+1}^{l-2}[(j+2)^{m-1}-j^{m-1}]I_{j}(t)\\
          &\leq[(L+1)^{m-1}+(L+2)^{m-1}]v_{0}+2(m-1)\sum_{j=L+1}^{l-2}(j+2)^{m-2}I_{j}(t).
\end{align*}
Using now (\ref{exponential decay}), we obtain
\begin{align*}
  \int_{L}^{l}v^{m}f(v,t)\der v\leq[(L+1)^{m-1}+(L+2)^{m-1}]v_{0}+2(m-1)v_{0}\sum_{j=L+1}^{l-2}\textup{e}^{(m-2)\ln{(j+2)}-Cj^{\gamma-1}(T-t)}.
\end{align*}
If we have that $t<T<T_{\textup{gel}}$, then we can find a constant $C(t,T)>0$ such that $m\ln{(j+2)}-C(T-t)j^{\gamma-1}\leq 0$, for every $j\geq C(t,T)$. This will give us an upper bound for $\int_{C(t,T)}^{\infty}v^{m}f(v,t)\der v$ and we can control $\int_{0}^{C(t,T)}v^{m}f(v,t)\der v$ from above using the fact that the first order moment is bounded. Thus, we conclude as in \cite[Proof of Lemma 9.2.2]{bookcoagulation} that for every $T<T_{\textup{gel}}$
estimate \eqref{momentestimate} holds.
 \end{proof}
 \begin{rmk}\label{remark test will all moments}
     Due to Proposition \ref{all large moments finite}, we have that we can test equation (\ref{weak formulation equation solution}) with $\varphi(v)=v^{m}$ or $\varphi(v)=(v-R)^{m}_{+}$ for all times before $T_{\textup{gel}}$. This is since we can make use  of Lebesgue's dominated convergence theorem.
 \end{rmk}

 We are now in a position to derive a contradiction for the assumption that $T_{\textup{gel}}>0$.
 
 \begin{prop}\label{big time small time}
 Let $T<T_{\textup{gel}}$. There exist $R_{0}>0$, which depends on $T>0$,  and a constant $C>0$, which does not depend on $m$, such that 
     \begin{align}\label{eqref for contradiction}
         T\leq t +C\bigg(\int_{R}^{\infty}(v-R)^{m}f(v,t)\der v\bigg)^{-k_{m}},
     \end{align}
     for every $t\in[0,T]$, for every $m>2$, where $k_{m}=\frac{\gamma-1}{m-1}$, and for every $R\geq R_{0}$. 
 \end{prop}
 \begin{proof}
 Due to Proposition \ref{all large moments finite}, we can test with $\varphi(v)=(v-R)^{m}_{+}$ in (\ref{weak formulation equation solution}) as mentioned in Remark \ref{remark test will all moments}. Let $0\leq s<t\leq T<T_{\textup{gel}}$, then we have
 \begin{align}\label{derive contradiction}
     \int_{R}^{\infty}(v-R)^{m}f(v,t)\der v&- \int_{R}^{\infty}(v-R)^{m}f(v,s)\der v\nonumber\\
     &=\frac{1}{2}\int_{s}^{t}\int_{(0,\infty)}\int_{(0,\infty)}K(v,v')f(v,\tau)f(v',\tau)\Theta_{\varphi}(v,v')\der v'\der v \der \tau\nonumber\\
     &\geq \frac{1}{2}\int_{s}^{t}\int_{(0,\frac{R}{2}]}\int_{[R,\infty)}K(v,v')f(v,\tau)f(v',\tau)\Theta_{\varphi}(v,v')\der v'\der v \der \tau,
 \end{align}
 since $\Theta_{\varphi}(v,v')\geq 0$ in this case. We have that $\varphi(v)=0$ in the region $(v,v')\in(0,\frac{R}{2}]\times[R,\infty)$ and thus (\ref{derive contradiction}) becomes
 \begin{align*}
     \int_{R}^{\infty}(v-R)^{m}&f(v,t)\der v- \int_{R}^{\infty}(v-R)^{m}f(v,s)\der v\\
     &\geq \frac{1}{2}\int_{s}^{t}\int_{(0,\frac{R}{2}]}\int_{[R,\infty)}K(v,v')f(v,\tau)f(v',\tau)[\varphi(v+v')-\varphi(v')]\der v'\der v \der \tau.
 \end{align*}
We have $(v+v'-R)^{m}-(v'-R)^{m}=m\int_{v'-R}^{v+v'-R}s^{m-1}\der s\geq m(v'-R)^{m-1}v, $ for $v'>R$ and $m>2$. Thus, we further deduce that
  \begin{align*}
     \int_{R}^{\infty}(v-R)^{m}&f(v,t)\der v- \int_{R}^{\infty}(v-R)^{m}f(v,s)\der v\\
     &\geq \frac{m}{2}\int_{s}^{t}\int_{(0,\frac{R}{2}]}\int_{[R,\infty)}K(v,v')f(v,\tau)f(v',\tau)(v'-R)^{m-1}v\der v'\der v \der \tau.
 \end{align*}
 From (\ref{rain kernel general}), (\ref{constant H}) and (\ref{definition of g}), we observe that $K(v,v')\geq Cv'^{\gamma}$ in the region $(v,v')\in(0,\frac{R}{2}]\times[R,\infty)$, if $R\geq H$, where $H$ is as in (\ref{constant H}), as proven in (\ref{lower bound kernel finite moment estimates}), which implies that
  \begin{align}
     \int_{R}^{\infty}(v-R)^{m}&f(v,t)\der v- \int_{R}^{\infty}(v-R)^{m}f(v,s)\der v\nonumber\\
     &\geq C m\int_{s}^{t}\int_{(0,\frac{R}{2}]}\int_{[R,\infty)}v'^{\gamma}f(v,\tau)f(v',\tau)(v'-R)^{m-1}v\der v'\der v \der \tau.\label{mass bounded below}
 \end{align}
 By Dini's monotone convergence theorem, as in (\ref{lower bound jr}), we have that there exists $R(T)>0$, such that, for all $R\geq R(T)$, we have that 
 \begin{align}\label{lower bound mass for gelation}
\int_{0}^{\frac{R}{2}}vf(v,t)\der v\geq\frac{v_{0}}{2}.
 \end{align} 
 Using this in (\ref{mass bounded below}), we obtain
 
   \begin{align*}
     \int_{R}^{\infty}(v-R)^{m}&f(v,t)\der v- \int_{R}^{\infty}(v-R)^{m}f(v,s)\der v\nonumber\\
     &\geq C m v_{0}\int_{s}^{t}\int_{[R,\infty)}v'^{\gamma}f(v',\tau)(v'-R)^{m-1}\der v' \der \tau\\
         &\geq C m v_{0}\int_{s}^{t}\int_{[R,\infty)}f(v',\tau)(v'-R)^{m+\gamma-1}\der v' \der \tau,
 \end{align*}
 for all $R\geq\max \{ R(T),H\}$, where $H$ is as in (\ref{constant H}) and $R(T)$ is as in (\ref{lower bound mass for gelation}). By H\"{o}lder's inequality, we obtain
 \begin{align}
     \int_{R}^{\infty}(v-R)^{m}f(v,t)\der v&\leq   \bigg(\int_{R}^{\infty}(v-R)^{m+\gamma-1}f(v,t)\der v \bigg)^{\frac{m-1}{m+\gamma-2}} \bigg(\int_{R}^{\infty}(v-R)f(v,t)\der v\bigg)^{\frac{\gamma-1}{m+\gamma-2}}\nonumber\\
     &\leq   \bigg(\int_{R}^{\infty}(v-R)^{m+\gamma-1}f(v,t)\der v \bigg)^{\frac{m-1}{m+\gamma-2}} v_{0}^{\frac{\gamma-1}{m+\gamma-2}}\label{holder holder}
 \end{align}
 and thus from (\ref{mass bounded below}) and (\ref{holder holder}), it follows that
    \begin{align}\label{desired behavior 1}
     \int_{R}^{\infty}(v-R)^{m}&f(v,t)\der v- \int_{R}^{\infty}(v-R)^{m}f(v,s)\der v\nonumber\\
     &\geq C m v_{0}v_{0}^{-\frac{\gamma-1}{m-1}}\int_{s}^{t}\bigg(\int_{[R,\infty)}f(v',\tau)(v'-R)^{m}\der v'\bigg)^{\frac{m+\gamma-2}{m-1}}\der \tau\nonumber\\
       &\geq C m v_{0}(1+v_{0})^{-(\gamma-1)}\int_{s}^{t}\bigg(\int_{[R,\infty)}f(v',\tau)(v'-R)^{m}\der v'\bigg)^{\frac{m+\gamma-2}{m-1}}\der \tau.
 \end{align}
 Denoting by $M_{R,m}:=\int_{R}^{\infty}(v-R)^{m}f(v,t)\der v$, we can now repeat the argument in \cite[Lemma 9.2.5]{bookcoagulation} to deduce that for every $0\leq t<T<T_{\textup{gel}}$, there exists $R(T)$ such that
 \begin{align}\label{desired behavior 2}
    T-t\leq C\frac{m-1}{m}M_{R,m}(t)^{-\frac{\gamma-1}{m-1}}\leq C M_{R,m}(t)^{-\frac{\gamma-1}{m-1}},
 \end{align}
 for every $R\geq R(T)$. The argument consists in using (\ref{desired behavior 1}) to find a suitable ODE, that gives us the desired behavior of $M_{R,m}$ in (\ref{desired behavior 2}). For this, notice that $\frac{m+\gamma-2}{m-1}>1$. 
 
 This concludes our proof.
 \end{proof}
 We are now in a position to prove Theorem \ref{theorem instantaneous gelation}.
 \begin{proof}[Proof of Theorem \ref{theorem instantaneous gelation}]
 Assume by contradiction that $T_{\textup{gel}}>0$ and let $0<T<T_{\textup{gel}}$. From  Corollary \ref{corollary infinity} we have that for every $t>0$, it holds that  $\liminf_{R\rightarrow\infty}\liminf_{m\rightarrow \infty} M_{R,m}(f(t))^{\frac{1}{m}}=\infty.$ Let $0<t<T<T_{\textup{gel}}$ in  Proposition \ref{big time small time} and then let first $m\rightarrow\infty$ and then $R\rightarrow\infty$ in \eqref{eqref for contradiction}. It follows that $T\leq t$, which gives the desired contradiction.
 \end{proof}

\subsection*{Acknowledgements}
The authors would like to thank B. Perthame for fruitful discussions about gelation theory and for explaining the arguments in \cite{gelfournier2025}. 

I.C. was supported by the European Research Council (ERC) under the European Union's Horizon 2020 research and innovation program Grant No. 637653, project BLOC ``Mathematical Study of Boundary Layers in Oceanic Motion'' and by the project BOURGEONS ``Boundaries, Congestion and Vorticity in Fluids: A connection with environmental issues '', grant ANR-23-CE40-0014-01 of the French National Research Agency (ANR).

B.N. and J.J.L.V. gratefully acknowledge the financial support of the collaborative
research centre The mathematics of emerging effects (CRC 1060, Project-ID 211504053) and of the Hausdorff Center for Mathematics (EXC 2047/1, Project-ID 390685813) funded through the Deutsche Forschungsgemeinschaft (DFG, German Research Foundation).

\subsubsection*{Statements and Declarations}

\textbf{Conflict of interest} The authors declare that they have no conflict of interest.

\textbf{Data availability} Data sharing not applicable to this article as no datasets were generated or analysed during the current study.

\printbibliography

\end{document}